\documentclass[12pt,reqno]{amsart}

\usepackage{amscd}
\usepackage{amsfonts}
\usepackage{amsmath}
\usepackage{amssymb}
\usepackage{amsthm}
\usepackage{fancyhdr}
\usepackage{latexsym}
\usepackage[colorlinks=true, pdfstartview=FitV, linkcolor=blue,
            citecolor=blue, urlcolor=blue]{hyperref}

\input epsf
\input texdraw
\input txdtools.tex
\input xy
\xyoption{all}

\usepackage{color}

\definecolor{Red}{rgb}{1.00, 0.00, 0.00}
\definecolor{DarkGreen}{rgb}{0.00, 1.00, 0.00}
\definecolor{Blue}{rgb}{0.00, 0.00, 1.00}
\definecolor{Cyan}{rgb}{0.00, 1.00, 1.00}
\definecolor{Magenta}{rgb}{1.00, 0.00, 1.00}
\definecolor{DeepSkyBlue}{rgb}{0.00, 0.75, 1.00}
\definecolor{DarkGreen}{rgb}{0.00, 0.39, 0.00}
\definecolor{SpringGreen}{rgb}{0.00, 1.00, 0.50}
\definecolor{DarkOrange}{rgb}{1.00, 0.55, 0.00}
\definecolor{OrangeRed}{rgb}{1.00, 0.27, 0.00}
\definecolor{DeepPink}{rgb}{1.00, 0.08, 0.57}
\definecolor{DarkViolet}{rgb}{0.58, 0.00, 0.82}
\definecolor{SaddleBrown}{rgb}{0.54, 0.27, 0.07}
\definecolor{Black}{rgb}{0.00, 0.00, 0.00}
\definecolor{dark-magenta}{rgb}{.5,0,.5}
\definecolor{myblack}{rgb}{0,0,0}
\definecolor{darkgray}{gray}{0.5}
\definecolor{lightgray}{gray}{0.75}

\usepackage[pdftex]{graphicx}
\usepackage{epstopdf}
 
\def\refer #1\par{\noindent\hangindent=\parindent\hangafter=1 #1\par}

\theoremstyle{plain}  
\newtheorem{theorem}{Theorem}
\newtheorem{proposition}{Proposition}
\newtheorem{lemma}{Lemma}
\newtheorem{corollary}{Corollary}

\theoremstyle{definition}

\begin{document}


\title{The fiberwise intersection theory}
\author{Gun Sunyeekhan}
\date\today
\maketitle

\parindent0in
\parskip0.1in


%

\begin{abstract}

We define a bordism invariant for  the fiberwise intersection theory. Under some certain conditions, this invariant is an obstruction for the theory. 
\end{abstract}

\setcounter{equation}{0}
 
\section{Introduction}

We start with the following assumptions for the intersection theory;

\begin{itemize}
\item[{\it(i)}] Let $P^p$ and $M^m$ be smooth manifolds.  Suppose $Q^{q} \subseteq  M^{m}$ is a closed submanifold and $f:P\to M$ smooth map such that $f \pitchfork  Q$ in $M$.

\item[{\it(ii)}]  Let $E(f,i_Q) :=$ the homotopy pullback of  $[P\stackrel{f}{\to} M \stackrel{i_Q}{\hookleftarrow} Q]$.
We have a smooth map $f^{-1}(Q) \to E(f,i_Q)$ and the bundle data determines an element in $\Omega^{fr}_{p+q-m}(E(f,i_Q)).$  

\item[{\it(iii)}] Let $N^{p+q-m} $ be a submanifold of $P$ and let $N \to E(f,i_Q)$ be a map which represents the same such an element in $\Omega^{fr}_{p+q-m}(E(f,i_Q)).$  
\end{itemize}

In 1974, Allan Hatcher and Frank Quinn \cite{H-Q}  showed in their work that if $f$ is an immersion and assume $\displaystyle m> q +\frac{p}{2} +1, m> p +\frac{q}{2} +1$, then we can homotope a map $f$ the the new map $g$  so that  $g^{-1}(Q) = N$. We develop their result to the case where $f$ is any smooth map and also weaken the dimension condition as follows; (See \cite{Gsun} for more details.)

\begin{theorem}{(Classical version)}
Given a smooth map $f:P^p \to M^m$ and $Q^q$ is a closed submanifold of $M$.  Assume  $\displaystyle m> q +\frac{p}{2} +1$. Then there is a map $g$ homotopic to $f$ such that $g^{-1}(Q) = N$.
\end{theorem}
 
Our proof is using the bundle data to construct the required homotopy step by step.  In this paper, we proceed along the same lines as the proof of the classical case to get the result for the fiberwise case which we will describe it in the next section. 


\section{Intersection Theory. (Fiberwise version)}

{\bf (I)}     Suppose that $E^{p+k}_{P}, E^{q+k}_{Q}$ and $E^{m+k}_{M}$ are smooth fiber bundles over a compact  manifold $B^{k}$. Let  $f:E^{p+k}_{P} \to E^{m+k}_{M}$  be a bundle map and $E_{Q}$ be a subbundle  of $E_{M}$ with the inclusion bundle map $i_{Q}: E_{Q} \hookrightarrow E_{M}$.

We have a commutative diagram
\begin{equation}
\xymatrix{ P^{p}\ar[d] & M^{m} \ar[d] & Q^{q}\ar[d] \\ E_{P} \ar[r]^{f} \ar[dr]_{pr_{P}} &E_{M}\ar[d]_<<<{pr_{M}} &E_{Q} \ar@{_{(}->}[l]_{i_{Q}} \ar[dl]^{pr_{Q}} \\ &B& }
\end{equation}

where $P,Q$ and $M$ are the fibers of $pr_{P},pr_{Q}$ and $pr_{M}$, respectively.

We may  assume that $f\pitchfork E_{Q}$ in $E_{M}$ ( See \cite{Koz}).

The homotopy pullback  is 

$E(f,i_{Q}) := \{ (x,\lambda,y) \in E_{P} \times E_{M}^{I} \times E_{Q} \  | \  \lambda(0) = f(x) , \lambda(1) = y\}$.

We have a diagram which commutes up to homotopy

\begin{equation}
\xymatrix{E(f,i_{Q}) \ar[r]^{\pi_Q} \ar[d]_{\pi_P} & E_{Q} \ar@{^{(}->}[d]_{i_{Q}} \\ E_{P} \ar[r]^f &E_{M}}
\end{equation}

where $\pi _{P}$ and $\pi_{Q}$  are the trivial projections, i.e. we have a homotpy


$E(f,i_Q)\times I  \stackrel{K}{\to} E_M$ defined by $K(x,\lambda,y,t) = \lambda(t)$.

We also  have a map
         $c:f^{-1}(E_{Q}) \to E(f,i_{Q})$ defined by  $x\mapsto (x,c_{f(x)},f(x))$ 
         
         where $c_{f(x)} =$ constant path in $E_M$ at $f(x).$ Note that $E(f,i_{Q})$ and $f^{-1}(E_{Q})$ are not necessarily the fiber bundles over $B$.
         
   
 Transversality yields a bundle map 
   
\begin{equation}
\xymatrix{\nu_{f^{-1}(E_{Q}) \subseteq E_{P}}\ar[r] \ar[d] &\nu_{E_{Q} \subseteq E_{M}} \ar[d] \\ f^{-1}(E_{Q}) \ar[r]^{f_{|f^{-1}(E_{Q})}} &E_{Q}}
\end{equation}
   
   Choose an embedding $\displaystyle E_{P}^{p+k} \subseteq S^{p+k+l}$, for sufficiently large $l$. Then we get   $f^{-1}(E_{Q}) \stackrel{i}{\hookrightarrow} E_{P}^{p+k} \subseteq S^{p+k+l}$. So  $\nu _{f^{-1}(E_{Q}) \subseteq S^{p+k+l}}  \cong \nu_{f^{-1}(E_{Q}) \subseteq E_{P}} \oplus i^{*} \nu_{E_{P} \subseteq S^{p+k+l}}$.

The commutative diagram

\begin{equation}
\xymatrix{f^{-1}(E_{Q}) \ar@/^1pc/[1,2]^{f_{|f^{-1}(E_{Q})}} \ar[1,1]^{c} \ar@/_1pc/[2,1]_{i} & & \\  & E(f,i_Q) \ar[r]^{\pi_{Q}} \ar[d]_{\pi_{P}} & E_Q \\  & E_P  & }
\end{equation}

yields a bundle map

\begin{equation}
\xymatrix{\nu_{f^{-1}(E_{Q}) \subseteq S^{p+k+l}}  \ar[d] \ar[r]^<<<<{\hat{c}}   & \pi^*(\nu_{E_{Q}\subseteq E_{M}}) \oplus \pi^*(\nu_{E_{P}\subseteq S^{p+k+l}}) := \xi \ar[d]  \\ f^{-1}(E_{Q}) \ar[r]^{c} & E(f,i_{Q}) }
\end{equation}

Thus $(c,\hat{c})$ determines an element  $[c,\hat{c}] \in \Omega^{fr}_{p+q+k-m}(E(f,i_{Q});  \xi)$.\\

{\bf (II)}  Suppose that  $(N \stackrel{c_{1}}{\to}E(f,i_Q), \nu_{N\subseteq S^{p+k+l}} \stackrel{\hat{c_1}}{\to} \xi)$ is another representative of  $[c,\hat{c}]$, where $N^{p+q+k-m} \subseteq E_P^{p+k} \subseteq S^{p+k+l}$. This means we have a normal bordism\\
 $(W \stackrel{\mathcal{C}}{\to} E(f,i_Q) , \nu_{W} \stackrel{\hat{\mathcal{C}}}{\to} \xi )$ between $(c,\hat{c})$ and $(c_{1}, \hat{c_{1}})$, i.e.  
 \begin{itemize}
 \item[({\it i})]$W^{p+q+k-m+1} \subseteq (S^{p+k+l} \times I)$,
 \item[({\it ii})] $\partial W \subseteq (S^{p+k+l} \times \partial I)$,
 \item[({\it iii})] $W \pitchfork ( S^{p+k+l} \times \partial I)$,
 \item[({\it iv})]$W \cap (S^{p+k+l} \times 0) = f^{-1}(E_Q)$ and $W \cap (S^{p+k+l} \times 1) = N $
 \end{itemize}

 such that 
 
 $\mathcal{C}_{| f^{-1}(E_Q)} =  c :f^{-1}(E_Q) \to E(f,i_{Q})\ \ , \ \ \mathcal{C}_{|N} = c_{1} : N \to E(f,i_Q)$
 
 $\hat{\mathcal{C}}_{|\nu_{f^{-1}(E_Q) \subseteq S^{p+k+l}}}= \hat{c}\hspace{3.1cm} , \ \ \hat{\mathcal{C}}_{|\nu_{N \subseteq S^{p+k+l}}} = \hat{c_{1}}.$\\
 






\begin{theorem}
\label{whitney}
Let $M^{m} $ and $N^n$ be smooth manifolds and $f:M \to N$ be a smooth map.
If $n > 2m$, then $f$ is homotopic to an embedding $g:M \to N$.
\end{theorem}
\begin{proof} 
See \cite{Muk}.
\end{proof}

\begin{lemma}
\label{relemb}
 Let $f:M^m \to N^n$ be a map between two smooth manifolds. Let $A$ be a closed submanifold of $M$. Assume that $f_{|A}$ is an embedding.

If $n > 2m$, then $f $ is homotopic to an embedding $ g$ relative to $A$.
\end{lemma}

\begin{proof}
  Let $T$ be a tubular neighborhood of $A$ in $M$.

{\bf Step I} Extend the embedding $f_{|A}:A \to N$ to an embedding $f_T:T \to N$ 

Let $\nu(A,M) $be the normal bundle of $A$ in $M$ and $D(\nu)$ denote the disc bundle of $\nu$. Then the tubular neighborhood theorem implies  $D(\nu(A,M)) \cong T$.

{\bf Claim :} For any given an embedding $A \stackrel{g}{\hookrightarrow} N$ and a vector bundle $\eta$ over $A$. Then

\begin{center}
$\bigg($ $g$ extends to an embedding of $D(\eta)$ into $N \bigg)$  

$\Longleftrightarrow$ $\bigg( $There exists a bundle monomorphism $\phi: \eta \to \nu(g) \bigg)$ 
 \end{center}
 
 where $\nu(g)$ is the normal bundle of $A$ in $N$ via $g$.

{\bf Proof Claim}  

$(\Leftarrow)$  We have a diagram 

\begin{equation}
\xymatrix{D(\eta) \ar[r]^{\phi} \ar[1,1] &D(\nu(g)) \ar[d]^{exp} &{\text{ zero section}}  \ar@{_{(}->}[l] \ar[d]^{\cong} \\ &N&A \ar@{_{(}->}[l]}
\end{equation}

where $exp$ is the exponential map.

Note that $exp(D(\nu(g)) \cong $ tubular neighborhood of $A$ in $N$ via $g$.
Then $exp \circ \phi : D(\eta) \hookrightarrow N$ is a desired embedding. 

$(\Rightarrow)$  Assume there exists an embedding $g_T$ so that  the following diagram commutes

\begin{equation}
\xymatrix{D(\eta) \ar[r]^{g_T} &N\\  A \ar@{_{(}->}[u]^i \ar[ur]_g&}
\end{equation}
 
 Then $\nu(g) \cong \eta \oplus i^* \nu(g_T)$.\\
 
 We are in the situation that we have a commutative diagram
 
 \begin{equation}
 \xymatrix{ A\ar[dr]^{f_{|A} = g} \ar@{^{(}->}[d]_i &\\ M \ar[r]_f &N}
 \end{equation}
 
 Let $\nu(f) := f^*\tau_N - \tau_M$. Then $i^*(\nu(f)) \oplus \nu(A,M) \stackrel{stable}{\cong} \nu(g)$.
 
 If $n -a > a$, then $i^*(\nu(f)) \oplus \nu(A,M) \cong \nu(g)$, so there exists a bundle monomorphism 
 
 \begin{equation}
 \xymatrix{ \nu(A,M) \ar[rr] \ar[dr] &&\nu(g) \ar[dl] \\& A&}
 \end{equation}

 Apply Claim when $g = f_{|A}$ and $\eta = \nu(A,M)$, then we have an extension embedding of $g$ from $D(\nu(A,M)) \cong T \stackrel{f_T}{\hookrightarrow} N$.\\
 
 {\bf Step 2} We have a map $f_{|M - int T} :M \setminus int(T) \to N$ and $\partial(M \setminus int(T)) = \partial T$.

$\displaystyle b > (c+\frac{a}{2} +1)$ and Theorem \ref{whitney}$ \Rightarrow$ $f_{|M - int(T)}$ is homotopic to an embedding $g_{M -int(T)}$.

Define $g:M \to N$ by  

\begin{equation*}
g(x) = \left\{
\begin{array}{rl}
g_{M-int(T)}(x)& \text{if } x \in M\setminus int(T)\\
g_T(x) & \text{if } x\in T.
\end{array} \right.
\end{equation*}

Then $f$ is homotopic to $g$ relative to $A$.

\end{proof}
\newpage

\begin{theorem} 
\label{main}
  Assume $\displaystyle m > q + (\frac{p+k}{2}) +1$. Then there exists a smooth map over $B$
  
  $$\Psi:E_P\times I \to E_M$$
  
 such that $\Psi_{|E_P\times \{0\}} = f$,  $\Psi \pitchfork E_Q$ and $\Psi_{|E_P\times \{1\}}^{-1}(E_Q) = N$.
 
  \end{theorem}
 
 Note that if we let $g = \Psi_{|E_P \times\{1\}}$. Then $g$ is fiber-preserving  homotopic to $f$ and $g^{-1}(E_Q) = N$.\\
\begin{proof} We divide the proof into $3$ steps,



%

{\bf Step 1:}   Goal: Homotope the map $W \stackrel{a: = \pi_{P} \circ \mathcal{C}}{\xrightarrow{\hspace*{1.2cm}}} E_P$ to an embedding over $B$. 

By assumption, we have
$$W \stackrel{\mathcal{C}}{\to} E(f,i_Q) \subseteq E_P\times E_M^{I} \times E_Q.$$

and we also have maps

\begin{center}

$W \stackrel{a: = \pi_{P} \circ \mathcal{C}}{\xrightarrow{\hspace*{1.2cm}}} E_P\ \ , \ \ W\stackrel{b:= \pi_{Q} \circ \mathcal{C}}{\xrightarrow{\hspace*{1.2cm}}} E_Q \ \ , \ \  W\times I \stackrel{H}{\to} E_M$  
\end{center}

where $H :=  K \circ (\mathcal{C} \times id_I)$, so $H_{|W\times 0} =  f\circ a ,  H_{|W \times 1} = b$.

Recall that  $\partial W = f^{-1}(E_Q) \sqcup N \ \ , \ \ a_{|\partial W}$ is just the inclusion of $f^{-1}(E_Q)$ and $N$ into $E_P$.



Apply the condition $\displaystyle m > q + \frac{p+k}{2} + 1$ to  Lemma $ \ref{relemb}$ , there exists an embedding  $\mathcal{A} \simeq a$ (rel $\partial W$), i.e. we have a commutative diagram

\begin{equation}
\xymatrix{W\times 1 \ar@{^{(}->}[d]  \ar@/^/[dr] ^{\mathcal{A}} & \\ W \times I \ar[r]^{L}  & E_P\\ W\times  0  \ar@{_{(}->}[u] \ar@/_/ [ur]_{a} & }
\end{equation}
  
We have a map  $W\stackrel{b:= \pi_{Q} \circ \mathcal{C}}{\xrightarrow{\hspace*{1.2cm}}} E_Q$. By concatenating the homotopy $H$ and $f\circ L$ together, we get a commutative diagram

$$\xymatrix{W\times 2 \ar@{^{(}->}[d] \ar[r]^{b} &E_{Q} \ar@{_{(}->}[d]^{i_{Q}} \\ W\times [1,2] \ar[r]^{V} &E_{M}\\ W\times 1 \ar@{_{(}->}[u] \ar[r]^{\mathcal{A}} &E_{P}\ar[u]_{f}}$$

Thus the following diagram commutes up to homotopy 

$$\xymatrix{W \ar[r]^b \ar@{^{(}->}[d]^{\mathcal{A}} &E_Q \ar@{^{(}->}[d]^{i_Q} \\ E_P \ar[r]^f &E_M}$$

Next, we want to modify the homotopy $V$ such that it is fiber preserving with respect to $pr_{M}$.


Note that we have a commutative diagram

\begin{equation}
\xymatrix{W\times 2 \ar[r]^b \ar@{^{(}->}[d] &E_Q \ar[d] ^{pr_Q}\\ W \times [1,2] \ar@{-->}[ur]^{V'} \ar[r]^<<<<<{pr_M \circ V} &B}
\end{equation}

We can apply the homotopy lifting property for $pr_Q$ to get a homotopy  of $b$ to $b'$ through $V'$ such that the following diagram commute

\begin{equation}
\label{con1}
\xymatrix{ W\times 2 \ar[r]^{b' := V'_{W\times1}} \ar@{^{(}->}[d] & E_Q \ar[dr]^{pr_Q} \ar@{^{(}->}[d]^{i_Q}&\\ W \times [1,2] \ar[r]^>>>>{V'} &E_Q \subseteq E_M \ar[r]^>>>>>{pr_M} &B\\  W \times 1 \ar@{_{(}->}[u]  \ar[r]^{\mathcal{A}} &E_P\ar[u]_{f}  \ar[u] _f  \ar[ur]_{pr_P} }
\end{equation}



Let  $\Psi_{W} := b'  :W \to E_Q$. Then $\Psi_{W}$ is a bundle map over $B$ through the lifting $V'$. 
 



{\bf Step 2} Goal: Construct a bundle isomorphism 
\begin{center}
$\nu(\mathcal{A}) \oplus \epsilon^1  \cong b'^{*}(\nu_{E_Q \subseteq E_M})$
\end{center}

where $\epsilon^1$ is the trivial bundle.

Since dim$W < $ rank $\nu(\mathcal{A})$, it is enough to give a stable equivalence between such bundles. 

Now, we have
\begin{equation}
\label{norm1}
W \stackrel{\mathcal{A}}{\hookrightarrow} E_P \ \subseteq S^{p+k+l} \Longrightarrow  
\nu_{W \subseteq S^{p+k+l}}   \cong  \nu(\mathcal{A}) \oplus \mathcal{A}^*(\nu_{E_P\subseteq S^{p+k+l} }) .
\end{equation}

We also have a commutative diagram

\begin{equation}
 \xymatrix{W \ar@/^2pc/[1,2]^{b'} \ar[1,1]^>>{\mathcal{C}} \ar@/_2pc/[2,1]_{\mathcal{A}}  & &\\  & E(f,i_Q) \ar[r]^{\pi_{Q}} \ar[d]_{\pi_{P}} & E_Q \\  &E_P&}
 \end{equation}

 

\begin{align}
\mathcal{A} \simeq a = \pi_P \circ \mathcal{C} \ \ &\Longrightarrow \mathcal{A}^{*}(\nu_{E_P \subseteq S^{p+k+l}}) \cong (\pi_P \circ \mathcal{C})^*(\nu_{E_P \subseteq S^{p+k+l}}).\label{norm2}\\
 b' \simeq b = \pi_Q \circ \mathcal{C} \ \ &\Longrightarrow b'^*(\nu_{E_Q\subseteq E_M}) \cong (\pi_Q \circ \mathcal{C})^*(\nu_{E_Q\subseteq E_M}).\label{norm3}
\end{align}


\noindent Thus the bundle map $\hat{\mathcal{C}}: \nu_{W \subseteq S^{p+k+l}\times I} \to \xi = \pi^*_P(\nu_{E_P\subseteq S^{p+k+l}}) \oplus \pi^*_Q(\nu_{E_Q\subseteq E_M})$ yields the following stable isomorphism
\begin{equation}
\label{norm4}
\nu_{W \subseteq S^{p+k+l}} \oplus  \epsilon^1 \stackrel{stable}{\cong} (\pi_{P} \circ \mathcal{C})^*(\nu_{E_P \subseteq S^{p+k+l}})  \oplus  (\pi_Q \circ \mathcal{C})^*(\nu_{E_Q \subseteq E_M}).
\end{equation}

\noindent Putting \eqref{norm1}, \eqref{norm2} \eqref{norm3}and \eqref{norm4} together, we get  
\begin{equation}
\nu(\mathcal{A})  \oplus (\pi_P \circ \mathcal{C})^*(\nu_{E_P \subseteq S^{p+k+l}})\oplus \epsilon^1 \stackrel{stable}{\cong} (\pi_{P} \circ \mathcal{C})^*(\nu_{E_P \subseteq S^{p+k+l}}) \oplus  b'^*(\nu_{E_Q \subseteq E_M}).
\end{equation}

\noindent Consequently, we have 
\begin{equation}
\label{stable}
\nu(\mathcal{A}) \oplus \epsilon^1  \cong b'^{*}(\nu_{E_Q \subseteq E_M}).
\end{equation}




%
\noindent This implies that we did construct a bundle map
\begin{equation}
\label{**}
\xymatrix{\nu(\mathcal{A} ) \oplus \epsilon^1\ar[r]^{\hat{b'}} \ar[d] & \nu_{E_Q \subseteq E_M} \ar[d] \\ W\ar[r]^{b' }  &E_Q}
\end{equation}
\noindent which give us the extension of map $b'$ to the tubular neighborhood of $W$ in $E_P$. More precisely, 

$$\Psi_T: D(\nu(\mathcal{A})) \hookrightarrow D(\nu(\mathcal{A}) \oplus \epsilon^1) \stackrel{\hat{b'}}{\to} D(\nu_{E_Q\subseteq E_M})$$
where $D$ denotes the disc bundle.Note that $\Psi_{T} \pitchfork E_Q$ and $\Psi_{T}(\partial D(\eta_1)) \subseteq E_M \setminus E_Q$. 

Since $\nu(\mathcal{A}) \oplus \epsilon^1  \cong b'^{*}(\nu_{E_Q \subseteq E_M})$, we can find a subbundle $\eta_2$ of $ b'^{*}(\nu_{E_Q \subseteq E_M})$ such that $\eta_2 \cong \nu(\mathcal{A}).$ For simplicity, let $\eta_1 := \nu(\mathcal{A})$.

\noindent {\bf Step 3.} Goal: Construct the smooth map $\Psi:E_P\times I \to E_M$ over $B$.





Recall that we have 
$$W \hookrightarrow D(\nu(\mathcal{A})) \simeq D(\nu(\mathcal{A}) \oplus \epsilon^1) \cong D(b'^*(\nu_{E_Q \subseteq E_M})).$$

Then there exists a neighborhood $\bar{D}$ of $W$ in  $D(b'^*(\nu_{E_Q \subseteq E_M}))$ such that $\bar{D} \simeq D(b'^*(\nu_{E_Q \subseteq E_M}))$ and $\bar{D} \cong D(\nu(\mathcal{A})) $.



According to $\eqref{con1}$, we have a commutative diagram

\begin{equation}
\xymatrix{W \times 2 \ar@{^{(}->}[d]  \ar[rr]^{V'_{|W \times 2}} && E_M\times 2 \ar@{_{(}->}[d] \\ W \times [1,2] \ar[rr]^{\Psi_1}&  & E_M \times [1,2] \\ W\times  1  \ar@{_{(}->}[u] \ar[rr]_{V'_{|W\times 1}} &&E_M \times 1 \ar@{^{(}->}[u] }
\end{equation}

where $\Psi_1(w,t) =  (V'(w,t), t)$.


Let $\eta_1 =  \nu(\mathcal{A} \oplus \epsilon^1), \eta_2 = b^*(\nu_{E_Q\subseteq E_M})$.

According to   $\eqref{**}$, there exists a bundle $\eta$ over $W \times I$ such that $\eta_{|W \times i} = \eta_i$  for $ i=1,2$ 

Let $D_1:= D(\nu(\mathcal{A}))$, $D_2 := \bar{D}$.

Then $D_i \hookrightarrow D(\eta_i)$ is a homotopy equivalence  for $i=1,2$ and also $D_1 \cong D_2$. 

Since $D_1\cup W\times [1,2] \cup D_2  \hookrightarrow D_1 \times [1,2]$ is a cofibration and a homotopy equivalence, there exist an extension $D_1 \times [1,2] \stackrel{\hat{\Psi}_1}{\to} E_M\times [1,2]$ such that the following diagram commutes

\begin{equation}
\xymatrix{D_2 \ar@{^{(}->}[d]  \ar[rr] && D(\nu_{E_Q\subseteq E_M}) \ar@{^{(}->}[r] ^{ \text{exp}} &E_M\times 2 \ar@{_{(}->}[d] \\ D_1 \times [1,2] \ar[rrr]^{\hat{\Psi}_1}& & & E_M\times [1,2] \\ D_1   \ar@{_{(}->}[u] \ar[rrr]^{f_{|D_1}} && &E_M\cong E_M\times 1 \ar@{^{(}->}[u] }
\end{equation}

Next we want to construct an embedding $W \stackrel{\mathcal{A}'}{\hookrightarrow}D_2 \times [2,3]$ such that the following hold:

\begin{itemize} 
\item[({\it i})] $\mathcal{A}'(W) \cap \{D_2 \times 2\} = f^{-1}(E_Q)$
\item[({\it ii})] $\mathcal{A}'(W) \cap \{D_2 \times 3\} = N$
\item[({\it iii})] $\mathcal{A}' \pitchfork D_2 \times \partial [2,3]$
\end{itemize}

We  start by letting $\alpha:W \to [2,3]$ be a smooth map such that $\alpha \pitchfork \partial[2,3]$, $\alpha^{-1}(2) = f^{-1}(E_Q)$ and $\alpha^{-1}(3) = N$, we also have an inclusion $W \stackrel{i_W}{\hookrightarrow} D_2$. 

Let $\mathcal{A}' := i_W \times \alpha$. Then $\mathcal{A}'$ is such a required map.  



By the construction, we have  $$D(\nu(\mathcal{A}')) \cong D_2 \times [2,3].$$

Let $\psi_2$ be the composition of the maps  $$W \stackrel{\mathcal{A}'}{\hookrightarrow} D_2 \times [2,3]  \stackrel{{\Psi_T}_{|D_2}  \times id_{[2,3]}}{\xrightarrow{\hspace*{1.5cm}}} M\times [2,3] \stackrel{proj}{\to} M $$

Define a map $\Psi_2:= \psi_2 \times \alpha: W  \to E_M\times [2,3]$.

Using the fact that $D(\nu(\mathcal{A}')) \cong D_2 \times [2,3]$, then  $D_2 \times \partial[2,3] \cup \mathcal{A}'(W)  \hookrightarrow D_2 \times [2,3]$ is a cofibration and homotopy equivalence. Hence there exist an extension $D_2 \times [2,3] \stackrel{\hat{\Psi}_2}{\to} E_M\times [2,3]$.

Note that for $(x,3) \in D_2 \times 3$ such that $\hat{\Psi}_2(x,3)  \in E_Q \times 3$, the map  $\hat{\Psi}_{2|D_2 \times 3} = \Psi_T $ forces that $x$ has to be in $W$, so by the definition of $\Psi_2$ implies $x\in N$. Thus

$\hat{\Psi}_{2|D_2 \times 3}^{-1}(E_Q) = N$.

We define a map $\displaystyle \tilde{\Psi}: \{ E_P \times I \}\cup \{D_1 \times [1, 2] \} \cup \{ D_2 \times [2,3] \}   \to E_M\times [0,3]$ by  
\begin{equation}
\tilde{\Psi}(p,t) =\left \{
\begin{array}{rl}
(f(p),t)  &\text{if } t \in [0,1]\\
(\hat{\Psi}_1(p), t) &\text{if } t \in [1,2]\\
(\hat{\Psi}_2(p), t) &\text{if } t \in [2,3].
\end{array} \right.
\end{equation}

Then $\tilde{\Psi}$ is well-defined map over $B$ by the construction.

It's not hard to see that  $ \{ E_P \times I \}\cup \{D_1 \times [1, 2] \} \cup \{ D_2 \times [2,3] \} $ is diffeomorphic to $E_P \times I$ . Define the map $\Psi$ to be the composition of maps $$ E_P \times I  \stackrel{\cong}{\to } \{ E_P \times I \}\cup \{D_1 \times [1, 2] \} \cup \{ D_2 \times [2,3] \} \stackrel{\tilde{\Psi}}{\to} E_M\times [0,3] \stackrel{proj}{\to} E_M$$
where $proj$ is the projection to the first factor.

Thus , we get a map $\Psi:E_P \times I \to E_M$ over $B$ so that $\Psi_{|E_P\times 0} = f$. By construction, $\Psi \pitchfork E_Q$ and $\Psi_{|E_P\times 1}^{-1}(E_Q) = N$ as required.
 
\end{proof} 

\begin{corollary} 
\label{obstruction}
  Assume $\displaystyle m > q + (\frac{p+k}{2}) +1$. Then we can fiber-preserving homotope a map $f$ to the map that its image does not intersect  $ E_Q $ if and only if 
 $[c,\hat{c}]  = 0 \in \Omega^{fr}_{p+q+k-m}(E(f,i_{Q});  \xi)$.

\end{corollary}

\section{Application to fixed point theory.}
Let $p:M^{m+k} \to B^k$ be a smooth fiber bundle with compact fibers and $k>2$. Assume that $B$ is a closed manifold. Let   $f:M \to M$  be a smooth map over $B$, i.e. $p\circ f = p$. 

\noindent The fixed point set of $f$ is  $$Fix(f) := \{ x \in M \ \ | \ \ f(x) = x \}.$$

We have a homotopy pull-back diagram

\begin{equation}
\xymatrix{ \mathcal{L}_fM \ar[r]^{ev_0} \ar[d]_{ev_1} &M \ar[d]^{\triangle} \\ M \ar[r]_>>>>{\triangle_f}
& M \times_B M}
\end{equation}

\noindent where
\begin{itemize}
\item[({\it i })] $\mathcal{L}_fM : = \{ \alpha \in M^I \ \ | \ \ f(\alpha(0)) = \alpha(1) \}$,
\item[({\it ii})] $ev_0$ and $ev_1$ are the evaluation map at $0$ and $1$ respectively,
\item[({\it iii})] $\triangle : =$ the diagonal map, defined by $x \mapsto (x,x)$,
\item[({\it iv})] $\triangle_f :=$ the twisted diagonal map, defined by $x \mapsto (x,f(x))$,
\item[({\it v})] $M \times_B M:=$ the fiber bundle over $B$ with fiber over $b\in B$, given by $F_b \times F_b$ where $F_b$ is the fiber of $p$ over $b$.
\end{itemize}

\newpage

\begin{proposition}{}
\label{triangle}
There exists a homotopy from $f$ to $f_1$ such that $\triangle_{f_1} \pitchfork \triangle$.
\end{proposition}
 
 The proof relies on the work of Kozniowski \cite{Koz}, relating to $B$-manifolds. Let $B$ be a smooth manifold. A $B$-manifold is a manifold $X$ together with a locally trivial submersion $p:X \to B$. A $B$-map is a smooth fiber-preserving map.
 \begin{lemma}{}
 \label{Kozinowski}
Let $X$ and $Y$ be $B$-manifolds and $Z$ be a $B$-submanifold of $Y$. Let $g:X \to Y$ be a $B$-map.Then there is a fiber-preserving smooth $B$-homotopy $H_t:X \to Y$ such that $H_0= g$ and $H_1 \pitchfork Z$.
\end{lemma}
 
\begin{proof} See \cite{Cou} for the proof.
\end{proof}

We have a transversal (pullback) square

\begin{equation}
\xymatrix{ Fix(f_1) \ar[r]^i \ar[d]_i &M \ar[d]^{\triangle}\\ M \ar[r]_>>>>>{\triangle_{f_1}} &M\times M}
\end{equation}

\noindent where $i$ is the inclusion. Transversality yields that $\nu(i) \cong i^*(\nu(\triangle)) \cong i^*(\tau M)$.

Choose an embedding $M \hookrightarrow S^{m+k}$ for sufficiently large $k$. Then we have 
$$\nu_{Fix(f_1) \subseteq S^{m+k}} \cong \nu(i) \oplus i^*(\nu_{M\subseteq S^{m+k}}) \cong i^*(\tau M) \oplus  i^*(\nu_{M\subseteq S^{p+k}}) \cong \epsilon.$$

\noindent We denote this bundle isomorphism by  $\nu_{Fix(f_1) \subseteq S^{m+k}}  \stackrel{\hat{g}}{\to} \epsilon$. We also have a map $Fix(f_1) \stackrel{g}{\to} \mathcal{L}_fM$ defined by $x \mapsto c_x$, where $c_x$ is the constant map at $x$. Thus $L^{bord}(f) := [Fix(f_1),g,\hat{g}]$ determines the element in $\Omega^{fr}_k(\mathcal{L}_fM ; \epsilon)$.

 Applying Theorem \ref{main}, We obtain the following corollary;

\begin{corollary}{(Converse of fiberwise Lefschetz fixed point theorem)}
\label{fixedmain}

Let  $f:M^{m+k}\to M^{m+k}$ be a smooth bundle map over the closed manifold $B^k$.  Assume that $m \geq k+3$. Then  $f$ is fiber homotopic to a fixed point free map if and only if $L^{bord}(f) = 0 \in \Omega^{fr}_k(\mathcal{L}_fM ; \epsilon).$
\end{corollary}

\renewcommand\bibname{REFERENCES}
\bibliographystyle{plain}
\bibliography{thesis_master_bib}

\end{document}